\documentclass[11pt,a4paper,reqno,twoside]{article}
\usepackage{hyperref}
\usepackage{amsfonts}
\usepackage{bbding}
\usepackage[all]{xy}
\usepackage{young}
\usepackage{multirow}
\usepackage{booktabs}
\usepackage{amsmath,amscd,amssymb,latexsym,srcltx,indentfirst,titlesec}
\usepackage{amsthm}

\usepackage{enumitem}

\setitemize[1]{itemsep=0pt,partopsep=0pt,parsep=\parskip,topsep=5pt}

\numberwithin{equation}{section}
\newtheorem{theorem}{Theorem}[section]
\newtheorem{lemma}[theorem]{Lemma}

\newtheorem{corollary}[theorem]{Corollary}
\newtheorem{remark}[theorem]{Remark}

\newtheorem{conjecture}[theorem]{Conjecture}
\newtheorem{definition}[theorem]{Definition}

\newtheorem*{thm*}{Main Theorem}
\allowdisplaybreaks

\begin{document}

\makeatletter

\newdimen\bibspace
\setlength\bibspace{2pt}
\renewenvironment{thebibliography}[1]{%
 \section*{\refname 
       \@mkboth{\MakeUppercase\refname}{\MakeUppercase\refname}}%
     \list{\@biblabel{\@arabic\c@enumiv}}%
          {\settowidth\labelwidth{\@biblabel{#1}}%
           \leftmargin\labelwidth
           \advance\leftmargin\labelsep
           \itemsep\bibspace
           \parsep\z@skip     %
           \@openbib@code
           \usecounter{enumiv}%
           \let\p@enumiv\@empty
           \renewcommand\theenumiv{\@arabic\c@enumiv}}%
     \sloppy\clubpenalty4000\widowpenalty4000%
     \sfcode`\.\@m}
    {\def\@noitemerr
      {\@latex@warning{Empty `thebibliography' environment}}%
     \endlist}

\makeatother

\pdfbookmark[2]{A note on Oliver's $p$-group conjecture}{beg}

\title{{\bf A note on Oliver's $p$-group conjecture} }

\footnotetext{Supported by NSFC grant (11501183).\\
1. Department of Mathematics, Hubei University, Wuhan, $430062$, China; \\
2. Departament de Matem$\mathrm{\grave{a}}$tiques, Universitat Aut$\mathrm{\grave{o}}$noma de Barcelona, E-08193 Bellaterra,
Spain.\\
\hspace*{2 ex}E-mail address: xuxingzhong407@126.com; xuxingzhong407@mat.uab.cat}

\author{{\small{ Xingzhong Xu$^{1,2}$} }
}

\date{29/03/2017}
\maketitle

{\small

\noindent\textbf{Abstract.} {\small{Let $S$ be a $p$-group for an odd prime $p$, Oliver proposed the conjecture that
the Thompson subgroup $J(S)$ is always contained in the Oliver subgroup $\mathfrak{X}(S)$.
That means he conjectured that $|J(S)\mathfrak{X}(S):\mathfrak{X}(S)|=1$.
Let $\mathfrak{X}_1(S)$ be a subgroup of $S$ such that $\mathfrak{X}_1(S)/\mathfrak{X}(S)$ is the center of $S/\mathfrak{X}(S)$.
In this short note, we prove that $J(S)\leq \mathfrak{X}(S)$
if and only if $J(S)\leq \mathfrak{X}_1(S)$.
 As an easy application, we prove that
 $|J(S)\mathfrak{X}(S):\mathfrak{X}(S)|\neq p$.
 }}\\

\noindent\textbf{Keywords: }{\small {Oliver's $p$-group;  the Thompson subgroup. }}

\noindent\textbf{Mathematics Subject Classification (2010):}  \ 20D15 $\cdot$ \ 20C20}

\section{\bf Introduction}

Recently, \cite{BLO1, BLO2} renewed the remarkable definition "centric linking system" which is like a
bridge between the saturated fusion systems and the classifying spaces.
The existence and uniqueness of centric linking system can be used to solve the Martino-Priddy conjecture.
In the beginning of these works, \cite{O} proposed Oliver's $p$-group conjecture, a purely group-theoretic problem.
A positive resolution of this conjecture would give the existence and uniqueness of centric linking systems for
fusion systems at odd primes(also see \cite{Ly}).

Curiously, the Martino-Priddy conjecture has been proved by \cite{O, O2}, and the existence and uniqueness of centric linking system
has been solved by \cite{Ch, O3}, but Oliver's $p$-group conjecture is still open.
There are only some works\cite{GHL, GHM, GHM2, Ly} about this conjecture for some
special cases.

Now, we list Oliver's $p$-group conjecture as follows.

\begin{conjecture}\cite[Conjecture 3.9]{O} Let $S$ be a $p$-group for an odd prime $p$. Then
$$J(S)\leq \mathfrak{X}(S),$$
where $J(S)$ is the Thompson subgroup generated by all elementary abelian $p$-subgroups whose rank is
the $p$-rank of $S$, and $\mathfrak{X}(S)$ is the Oliver subgroup described in \cite[Definition 3.1]{O} and \cite[Definition]{GHL}.
\end{conjecture}

Our main results are as follows.

\begin{theorem} Let $S$ be a finite $p$-group for an odd prime $p$.
Set $\mathfrak{X}_1(S) \leq S$ such that  $\mathfrak{X}_1(S)/\mathfrak{X}(S)=Z(S/\mathfrak{X}(S))$.
If $E$ is an elementary abelian $p$-subgroup whose rank is
the $p$-rank of $S$ and $E \leq \mathfrak{X}_1(S)$,
 then $E\leq \mathfrak{X}(S)$.
\end{theorem}

Here, $Z(S/\mathfrak{X}(S))$ means the center of $S/\mathfrak{X}(S)$. Moreover, using above theorem, we get the following corollary.
\begin{corollary} Let $S$ be a finite $p$-group for an odd prime $p$.
Set $\mathfrak{X}_1(S) \leq S$ such that $\mathfrak{X}_1(S)/\mathfrak{X}(S)=Z(S/\mathfrak{X}(S))$.
Then $J(S)\leq \mathfrak{X}(S)$ if and only if $J(S)\leq \mathfrak{X}_1(S)$.
\end{corollary}

Actually, if $S/\mathfrak{X}(S)$ is not trivial, then $\mathfrak{X}(S)$ is a proper subgroup of $\mathfrak{X}_1(S)$.
So this corollary should be useful to understand Oliver's $p$-group conjecture.

With further consideration, we get the following theorems.
\begin{theorem} Let $S$ be a finite $p$-group for an odd prime $p$. Then Oliver's $p$-group conjecture holds
 for $S$ if and only if $\mathfrak{X}(L)=\mathfrak{X}(S)$ for each subgroup $L$ with
 $\mathfrak{X}(S)\leq L\leq J(S)\mathfrak{X}(S)$.
\end{theorem}

\begin{theorem} Let $S$ be a finite $p$-group for an odd prime $p$.
If $L\unlhd S$ for each subgroup $L$ with $\mathfrak{X}(S)\leq L\leq J(S)\mathfrak{X}(S)$,
 then $J(S)\leq \mathfrak{X}(S)$.
\end{theorem}

$Structure~ of ~ the~ paper:$ After recalling the basic definitions and properties of Oliver's $p$-group in Section 2,
we give proofs of Theorem 1.2 and Corollary 1.3 in Section 3. Then in Section 4 we prove Theorem 1.4-5.

\section{\bf Oliver's $p$-group and some lemmas}

In this section we collect some contents that will be need later, we refer to \cite{GHL, GHM,O}.
First, we recall the definition of Oliver's $p$-group as follows.
\begin{definition}\cite[Definition 3.1]{O},\cite[Definition]{GHL} Let $S$ be a finite $p$-group and
$K\unlhd S$ a normal subgroup. There exists a sequence
$$1=Q_0\leq Q_1\leq\cdots\leq Q_n=\mathfrak{X}(S)$$
such that $Q_i\unlhd S$, and such that
$$[\Omega_1(C_S(Q_{i-1})), Q_i; p-1]=1\quad\quad \quad\quad \quad\quad (\ast)$$
holds for each $1\leq i\leq n$.
The unique largest normal subgroup of $S$ which admits such a sequence
is called $\mathfrak{X}(S)$, the Oliver subgroup of $S$.
\end{definition}

\begin{remark}\cite[p.334]{O} If $K, L\unlhd S$ and there exist two sequences
$$1=Q_0\leq Q_1\leq\cdots\leq Q_n=K;~~~1=R_0\leq R_1\leq\cdots\leq R_m=L$$
such that $Q_i\unlhd S, R_j\unlhd S$, and such that
$$[\Omega_1(C_S(Q_{i-1})), Q_i; p-1]=1~~~ \mathrm{and} ~~~[\Omega_1(C_S(R_{j-1})), R_j; p-1]=1$$
holds for each $1\leq i\leq n$ and $1\leq j\leq m$ respectively.
Then we have a sequence
$$1=Q_0\leq Q_1\leq\cdots\leq Q_n=K\leq Q_nR_0\leq Q_nR_1\leq\cdots\leq Q_nR_m=KL$$
satisfies the condition $(\ast)$.
\end{remark}

The following lemmas are important to prove the main results.
\begin{lemma}\cite[Lemma 3.2]{O} Let $S$ be a finite $p$-group, then $C_{S}(\mathfrak{X}(S))=Z(\mathfrak{X}(S))$.
\end{lemma}

\begin{lemma}\cite[Lemma 3.3]{O} Let $S$ be a finite $p$-group. Let $Q\unlhd S$ be any normal subgroup
such that
$$[\Omega_1(Z(\mathfrak{X}(S))), Q; p-1]=1.$$
Then $Q\leq \mathfrak{X}(S)$.
\end{lemma}

\section{\bf The proof of the Theorem 1.2}

In this section, we give a proof of Theorem 1.2. We prove the following lemma first.

\begin{lemma} Let $S$ be a finite $p$-group, if $\mathfrak{X}(S)\leq L\leq S$, then
$$\mathfrak{X}(S)\leq\mathfrak{X}(L).$$
\end{lemma}
\begin{proof} There exists a series of subgroups
$$1=Q_0\leq Q_1\leq\cdots\leq Q_n=\mathfrak{X}(S)$$
such that $Q_i\unlhd S$, and such that
$$[\Omega_1(C_S(Q_{i-1})), Q_i; p-1]=1$$
holds for each $1\leq i\leq n$.

Now, we will prove that $\mathfrak{X}(L)\geq\mathfrak{X}(S)$.
Since $Q_i\unlhd S$, thus $Q_i\unlhd L$. Also
$C_{L}(Q_i)\leq C_{S}(Q_i)$ for  each $1\leq i\leq n$.
Then we have
$$[\Omega_1(C_L(Q_{i-1})), Q_i; p-1]\leq[\Omega_1(C_S(Q_{i-1})), Q_i; p-1]=1.$$
Hence, by the definition of $\mathfrak{X}(L)$, we have $\mathfrak{X}(L)\geq\mathfrak{X}(S)$.
\end{proof}

Using the above lemma, we give a proof of Theorem 1.2 as follows.

\begin{theorem} Let $S$ be a finite $p$-group for an odd prime $p$.
Set $\mathfrak{X}_1(S) \leq S$ such that  $\mathfrak{X}_1(S)/\mathfrak{X}(S)=Z(S/\mathfrak{X}(S))$.
If $E$ is an elementary abelian $p$-subgroup whose rank is
the $p$-rank of $S$ and $E \leq \mathfrak{X}_1(S)$,
 then $E\leq \mathfrak{X}(S)$.
\end{theorem}

\begin{proof}Set $E$ is an elementary abelian $p$-subgroup whose rank is
the $p$-rank of $S$. Let $(S, E)$ be a minimal counterexample,
that is  $E \leq \mathfrak{X}_1(S)$,
but $E\nleq \mathfrak{X}(S)$.
We will consider the cases whether $E\mathfrak{X}(S)=S$ in the following.

\textbf{Case 1.} If $E\mathfrak{X}(S)=S$. That is $S/\mathfrak{X}(S)$ is abelian.
By \cite[Theorem 1.1]{GHM}, we have $J(S)\leq \mathfrak{X}(S)$. So $E\leq J(S)\leq \mathfrak{X}(S)$.
That is a contradiction.

\textbf{Case 2.}  If $E\mathfrak{X}(S)\lneq S$, set $L=E\mathfrak{X}(S)$.
Since $E \leq \mathfrak{X}_1(S)$, we have $L\leq \mathfrak{X}_1(S)$.
We assert that $\mathfrak{X}(L)=\mathfrak{X}(S)$.

Since $\mathfrak{X}(S)\leq L$, we have $\mathfrak{X}(L)\geq\mathfrak{X}(S)$
by Lemma 3.1. By the definition of Oliver's subgroup,
there exist two sequences
$$1=Q_0\leq Q_1\leq\cdots\leq Q_n=\mathfrak{X}(S);~~~1=R_0\leq R_1\leq\cdots\leq R_m=\mathfrak{X}(L)$$
such that $Q_i\unlhd S, R_j\unlhd L$, and such that
$$[\Omega_1(C_S(Q_{i-1})), Q_i; p-1]=1~~~ \mathrm{and} ~~~[\Omega_1(C_L(R_{j-1})), R_j; p-1]=1$$
holds for each $1\leq i\leq n$ and $1\leq j\leq m$ respectively.
we have a sequence
$$1=Q_0\leq Q_1\leq\cdots\leq Q_n=\mathfrak{X}(S)= Q_nR_0\leq Q_nR_1\leq\cdots\leq Q_nR_m=\mathfrak{X}(L)$$
satisfies the condition $(\ast)$ in Definition 2.1 for $L$.

If $\mathfrak{X}(S)\lneq \mathfrak{X}(L)$, then we can suppose that $Q_{n}=\mathfrak{X}(S)= Q_nR_0\lneq Q_nR_1$.
Since the above sequence satisfies the condition $(\ast)$ in Definition 2.1 for $L$,
it implies that
$$[\Omega_1(C_L(Q_{n})), Q_nR_1; p-1]=1.$$
Also $Q_{n}=\mathfrak{X}(S)$ and $L\geq \mathfrak{X}(S)$,
we have $C_L(Q_{n})=Z(\mathfrak{X}(S))$ by \cite[Lemma 3.2]{O}.
That means
$$[\Omega_1(Z(\mathfrak{X}(S))), Q_nR_1; p-1]=1.$$
Since $L\leq \mathfrak{X}_1(S)$, that is $L/\mathfrak{X}(S)\leq Z(S/\mathfrak{X}(S))$.
So we have $Q_nR_1\unlhd S$.
Hence, we have $Q_nR_1\leq \mathfrak{X}(S)$ by \cite[Lemma 3.3]{O}.
That is a contradiction. Hence, we have $\mathfrak{X}(S)=\mathfrak{X}(L)$.

Since $L=E\mathfrak{X}(S)\lneq S$, we can use induction to get $E\leq \mathfrak{X}(L)=\mathfrak{X}(S)$.
First, $E$ is also an elementary abelian $p$-subgroups whose rank is
the $p$-rank of $L$. And $\mathfrak{X}_1(L)=L$ because $\mathfrak{X}(L)=\mathfrak{X}(S)$.
So by induction, we have $E\leq \mathfrak{X}(L)$ because $E\leq L=\mathfrak{X}_1(L)$.
Hence $E\leq \mathfrak{X}(L)=\mathfrak{X}(S)$, that is a contradiction.
\end{proof}

As a corollary, we give a proof of Corollary 1.3 as follows.

\begin{corollary} Let $S$ be a finite $p$-group for an odd prime $p$.
Set $\mathfrak{X}_1(S) \leq S$ such that  $\mathfrak{X}_1(S)/\mathfrak{X}(S)=Z(S/\mathfrak{X}(S))$.
If $J(S)\leq \mathfrak{X}_1(S)$,
 then $J(S)\leq \mathfrak{X}(S)$.
\end{corollary}

\begin{proof} Since $J(S)$ is the Thompson subgroup generated by all elementary abelian $p$-subgroups whose rank is
the $p$-rank of $S$, we have $J(S)\leq \mathfrak{X}(S)$ by
 Theorem 3.2.
\end{proof}

Now, as an easy application, we give a proof of the following corollary.

\begin{corollary} Let $S$ be a finite $p$-group for an odd prime $p$.
Then  $|J(S)\mathfrak{X}(S):\mathfrak{X}(S)|\neq p$.
\end{corollary}

\begin{proof} Suppose $|J(S)\mathfrak{X}(S):\mathfrak{X}(S)|=p$,
since $J(S)\mathfrak{X}(S)/\mathfrak{X}(S)\unlhd S/\mathfrak{X}(S)$,
thus $$J(S)\mathfrak{X}(S)/\mathfrak{X}(S)\leq Z(S/\mathfrak{X}(S))$$
because $(J(S)\mathfrak{X}(S)/\mathfrak{X}(S))\cap Z(S/\mathfrak{X}(S))\neq 1$.
By Corollary 3.3, we have $J(S)\leq \mathfrak{X}(S)$.
That is a contradiction to the assumption, thus we have
$|J(S)\mathfrak{X}(S):\mathfrak{X}(S)|\neq p$.
\end{proof}

\section{\bf The proof of the Theorem 1.4}

Now, we prove Theorem 1.4 as follows.
\begin{theorem} Let $S$ be a finite $p$-group for an odd prime $p$. Then Oliver's $p$-group conjecture holds
 for $S$ if and only if $\mathfrak{X}(L)=\mathfrak{X}(S)$ for each subgroup $L$ with
 $\mathfrak{X}(S)\leq L\leq J(S)\mathfrak{X}(S)$.
\end{theorem}

\begin{proof} Suppose first Oliver's $p$-group conjecture holds
 for $S$, that is $J(S)\leq \mathfrak{X}(S)$. Then $J(S)\mathfrak{X}(S)=\mathfrak{X}(S)$.
 So we get the result.

 Conversely, let $S$ be a minimal counterexample,
 that is  $J(S)\nleq \mathfrak{X}(S)$. And, if $L$ is
a finite $p$-group with $|L|\lneq |S|$, and
$\mathfrak{X}(T)=\mathfrak{X}(L)$ for each subgroup $T$ with
 $\mathfrak{X}(L)\leq T\leq J(L)\mathfrak{X}(L)$,
then $J(L)\leq \mathfrak{X}(L)$.

Since $J(S)\nleq \mathfrak{X}(S)$,
thus we can let $E$ be an elementary abelian $p$-subgroup whose rank is
the $p$-rank of $S$ and $E\nleq \mathfrak{X}(S)$.
Set $H=E\mathfrak{X}(S)$, and
we will consider the cases whether $E\mathfrak{X}(S)=S$ in the following.

\textbf{Case 1.} If $H=E\mathfrak{X}(S)=S$, then $S/\mathfrak{X}(S)$ is abelian.
By \cite[Theorem 1.1]{GHM}, we have $J(S)\leq \mathfrak{X}(S)$.
That is a contradiction.

\textbf{Case 2.} If $H=E\mathfrak{X}(S)\lneq S$,
we will use induction to get $J(H)\leq \mathfrak{X}(H)$.
In fact, let $T\leq H$ such that $\mathfrak{X}(H)\leq T\leq J(H)\mathfrak{X}(H)$.
Since $\mathfrak{X}(S)\leq H\leq J(S)\mathfrak{X}(S)$,
by the condition of $S$, we have $\mathfrak{X}(H)=\mathfrak{X}(S)$.
Hence $\mathfrak{X}(S)=\mathfrak{X}(H)\leq T\leq J(H)\mathfrak{X}(H)\leq J(S)\mathfrak{X}(S).$
Also by the condition of $S$, we have $\mathfrak{X}(T)=\mathfrak{X}(S)=\mathfrak{X}(H)$.
Hence by induction, we have $J(H)\leq \mathfrak{X}(H)=\mathfrak{X}(S)$.

Since $E$ is an elementary abelian $p$-subgroup whose rank is
the $p$-rank of $S$, thus $E$ is also an elementary abelian $p$-subgroup whose rank is
the $p$-rank of $H$ because $E\leq H\leq S$.
Then $E\leq J(H) \leq \mathfrak{X}(H)= \mathfrak{X}(S)$.
That is also a contradiction.

So, we prove this theorem.
\end{proof}

Now, we take an application as follows.

\begin{theorem} Let $S$ be a finite $p$-group for an odd prime $p$.
If $L\unlhd S$ for each subgroup $L$ with $\mathfrak{X}(S)\leq L\leq J(S)\mathfrak{X}(S)$,
 then $J(S)\leq \mathfrak{X}(S)$.
\end{theorem}

\begin{proof}By Theorem 4.1, we only need to prove that
$\mathfrak{X}(L)=\mathfrak{X}(S)$ when $\mathfrak{X}(S)\leq L\leq J(S)\mathfrak{X}(S)$.

Since $\mathfrak{X}(S)\leq L$, we have $\mathfrak{X}(L)\geq\mathfrak{X}(S)$
by Lemma 3.1. By the definition of Oliver's subgroup,
there exist two sequences
$$1=Q_0\leq Q_1\leq\cdots\leq Q_n=\mathfrak{X}(S);~~~1=R_0\leq R_1\leq\cdots\leq R_m=\mathfrak{X}(L)$$
such that $Q_i\unlhd S, R_j\unlhd L$, and such that
$$[\Omega_1(C_S(Q_{i-1})), Q_i; p-1]=1~~~ \mathrm{and} ~~~[\Omega_1(C_L(R_{j-1})), R_j; p-1]=1$$
holds for each $1\leq i\leq n$ and $1\leq j\leq m$ respectively.
we have a sequence
$$1=Q_0\leq Q_1\leq\cdots\leq Q_n=\mathfrak{X}(S)= Q_nR_0\leq Q_nR_1\leq\cdots\leq Q_nR_m=\mathfrak{X}(L)$$
satisfies the condition $(\ast)$ in Definition 2.1 for $L$.

If $\mathfrak{X}(S)\lneq \mathfrak{X}(L)$, then we can suppose that $Q_{n}=\mathfrak{X}(S)= Q_nR_0\lneq Q_nR_1$.
Since the above sequence satisfies the condition $(\ast)$ in Definition 2.1 for $L$,
it implies that
$$[\Omega_1(C_L(Q_{n})), Q_nR_1; p-1]=1.$$
Also $Q_{n}=\mathfrak{X}(S)$ and $L\geq \mathfrak{X}(S)$,
we have $C_L(Q_{n})=Z(\mathfrak{X}(S))$ by \cite[Lemma 3.2]{O}.
That means
$$[\Omega_1(Z(\mathfrak{X}(S))), Q_nR_1; p-1]=1.$$
By the condition of the theorem, we have $Q_nR_1\unlhd S$.
So we have $Q_nR_1\leq \mathfrak{X}(S)$ by \cite[Lemma 3.3]{O}.
That is a contradiction. Hence, we have $\mathfrak{X}(S)=\mathfrak{X}(L)$.
\end{proof}

\textbf{ACKNOWLEDGMENTS}\hfil\break
The author would like to thank  Prof. C. Broto for his constant encouragement in Barcelona in Spain.

\end{document}